\documentclass[11pt]{amsart}
\usepackage{amsmath, amssymb,graphicx,mathrsfs,enumerate}
\usepackage[all]{xy}
\usepackage{amsthm}
\usepackage{amssymb}
\usepackage{latexsym}
\usepackage{longtable}
\usepackage{epsfig}
\usepackage{amsmath}
\usepackage{hhline}

\newcommand{\Out}{\operatorname{Out}}

\newtheorem{thm}{Theorem}[section]

\newtheorem{lem}[thm]{Lemma}

\theoremstyle{definition}

\newcommand{\la}{\langle}
\newcommand{\ra}{\rangle}

\newcommand{\Irr}{{\mathrm {Irr}}}
\newcommand{\cd}{{\mathrm {cd}}}
\newcommand{\Aut}{{\mathrm {Aut}}}

\newcommand{\Center}{{\mathbf {Z}}}

\newcommand{\PSL}{{\mathrm {PSL}}}

\newcommand{\PSp}{{\mathrm {PSp}}}

\newcommand{\PSU}{{\mathrm {PSU}}}

\newcommand{\PGL}{{\mathrm {PGL}}}

\newcommand{\Alt}{\mathrm{A}}
\newcommand{\SSS}{\mathrm{S}}
\newcommand{\Sym}{\mathrm{Sym}}

\newcommand{\N}{\mathbf{N}}

\begin{document}

\title[Conjugacy classes]{Finite groups have more conjugacy classes}

\author[B. Baumeister]{Barbara Baumeister}
\address{Fakult\"at f\"ur Mathematik, Universit\"at Bielefeld, Postfach 10 01 31,  D-33501 Bielefeld, Germany}
\email{b.baumeister@math.uni-bielefeld.de}

\author[A. Mar\'oti]{Attila Mar\'oti}
\address{Fachbereich
Mathematik, Technische Universit\"{a}t Kaiserslautern, Postfach
3049, 67653 Kaiserslautern, Germany \and Alfr\'ed R\'enyi Institute of
  Mathematics, Re\'altanoda utca 13-15, H-1053, Budapest, Hungary}
\email{maroti@mathematik.uni-kl.de \and maroti.attila@renyi.mta.hu}

\author[H. P. Tong-Viet]{Hung P. Tong Viet}
\address{Department of Mathematics and Applied Mathematics,
University of Pretoria,
Private Bag X20, Hatfield, Pretoria 0002,
South Africa}
\email{Hung.Tong-Viet@up.ac.za}
\date{24 Feb 2015}

\thanks{The second author was supported by an Alexander von Humboldt Fellowship for Experienced Researchers and by OTKA K84233.
Tong-Viet's work is based on the research supported in part by the National Research Foundation of South Africa (Grant Number 93408).
Part of the work was done while the last author held a position at the CRC 701 within the
project C13 `The geometry and combinatorics of groups'. The first and second authors also wish to thank the CRC 701 for its support.}



\begin{abstract}
We prove that for every $\epsilon > 0$ there exists a $\delta > 0$ so that every group of order $n \geq 3$ has at least $\delta \log_{2} n/{(\log_{2} \log_{2} n)}^{3+\epsilon}$ conjugacy classes. This sharpens earlier results of Pyber and Keller. Bertram speculates whether it is true that every finite group of order $n$ has more than $\log_{3}n$ conjugacy classes. We answer Bertram's question in the affirmative for groups with a trivial solvable radical.
\end{abstract}
\maketitle

\section{Introduction}
For a finite group $G$ let $k(G)$ denote the number of conjugacy classes of $G$. Answering a question of Frobenius, Landau \cite{Landau} proved in 1903 that for a given $k$ there are only finitely many groups having $k$ conjugacy classes. Making this result explicit, we have $\log \log |G| < k(G)$ for any non-trivial finite group $G$ (see Brauer \cite{Brauer}, Erd\H os and Tur\'an \cite{ET}, Newman \cite{Newman}). (Here and throughout the paper the base of the logarithms will always be $2$ unless otherwise stated.) Problem 3 of Brauer's list of problems \cite{Brauer} is to give a substantially better lower bound for $k(G)$ than this.

Pyber \cite{Py} proved that there exists a constant $\epsilon>0$ so that for every finite group $G$ of order at least $3$ we have $\epsilon \log |G|/{(\log \log |G|)}^{8} < k(G)$. Almost 20 years later Keller \cite{Keller} replaced the $8$ in the previous bound by $7$. Our first result gives a further improvement to Pyber's theorem.

\begin{thm}
\label{thm1}
For every $\epsilon > 0$ there exists a $\delta > 0$ so that for every finite group $G$ of order at least $3$ we have
$\delta \log |G|/{(\log \log |G|)}^{3+\epsilon} < k(G)$.
\end{thm}

There are many lower bounds for $k(G)$ in terms of $|G|$ for the various classes of finite groups $G$.
For example, Jaikin-Zapirain \cite{JZ} gave a better than logarithmic lower bound for $k(G)$ when $G$ is a nilpotent group. For supersolvable $G$ Cartwright \cite{C} showed $(3/5) \log |G| < k(G)$. For solvable groups the best bound to date is a bit worse than logarithmic and is due to Keller \cite{Keller}.

The conjecture whether there exists a universal constant $c > 0$ so that $c \log |G| < k(G)$ for any finite group $G$
has been intensively studied by many mathematicians including Bertram, see for instance \cite{Bertram}.
Bertram observed that $k(G) = \lceil \log_3(|G|) \rceil$ when
$G = \PSL_3(4)$ or $M_{22}$ and checked the proposed bound for certain  small groups \cite[p. 96]{Bertram06}.
He then speculates whether $\log_{3}|G| < k(G)$ is true for every finite group $G$.
 In our second result we answer Bertram's question in the
affirmative for groups with a trivial solvable radical.

\begin{thm}\label{th:base 3}
Let $G$ be a finite group with a trivial solvable radical. Then $\log_{3} |G| < k(G)$.
\end{thm}

The paper is structured as follows.  We prove Theorem \ref{thm1} in Section~\ref{sec:asymptotics}. This is done by first improving \cite[Lemma~4.7]{Py} which gives the lower bound for $\log k(G)$ in terms of $\log |G|$ for finite groups with a trivial solvable radical and then applying the argument in \cite{Py} and \cite{Keller} to get the required result for arbitrary finite groups. In Sections \ref{sec:c2}, we compute explicitly the constant $c_2$ arising from Lemma \ref{l2}. In Section \ref{sec:almost simple} we verify Theorem \ref{th:base 3} for some almost simple groups whose automorphism groups have a bounded number of orbits on their socles and finally the full proof of Theorem \ref{th:base 3} is carried out in Section \ref{sec:trivial radical}.

\section{Asymptotics}
\label{sec:asymptotics}

In this section we first improve \cite[Lemma 4.7]{Py}.

\begin{thm}
\label{th:asymptotics}
For every $\epsilon > 0$ there exists $\delta > 0$ so that for every non-trivial finite group $G$ with trivial solvable radical we have
$\delta \cdot (\log |G|)^{1/(3+\epsilon)} < \log k(G)$.
\end{thm}

We will prove Theorem \ref{th:asymptotics} in this section.
Let $G$ be a non-trivial finite group with trivial solvable radical. Suppose that $G$ has $r$ minimal normal subgroups $M_{1}, \ldots , M_{r}$. Then each $M_{i}$ with $1 \leq i \leq r$ is equal to a direct product $T_{i,1} \times \cdots \times T_{i,n_{i}}$ of $n_{i}$ isomorphic non-abelian simple groups $T_{i,j}$ with $1 \leq j \leq n_{i}$. Put $n = \sum_{i=1}^{r} n_{i}$, and let $N$ be the socle of $G$, that is, $M_{1} \times \cdots \times M_{r}$.

The group $G$ permutes the simple direct factors of each $M_{i}$ for $1 \leq i \leq r$. Let $B$ be the kernel of the action of $G$ on the set of $n$ simple direct factors of $N$. Then $B$ contains $N$ and $B/N$ embeds in the direct product of the outer automorphism groups of the $n$ simple direct factors of $N$. Furthermore $G/B$ is a subgroup of $\SSS_{n_1}\times\SSS_{n_2}\times\cdots\times\SSS_{n_r}\le \SSS_{n}.$


For a non-abelian finite simple group $T$ let $k^{*}(T)$ denote the number of $\mathrm{Aut}(T)$-orbits on $T$. By Burnside's theorem,
$|T|$ has at least $3$ different prime divisors, so $k^{*}(T) \geq 4$ by Cauchy's theorem.
Further, \cite[Lemma 2.5]{Py} and \cite[Lemma 4.4]{Py} yield the following.

\begin{lem}
\label{l1}
There exists a universal constant $c_{1} > 0$ so that whenever
$G$ is a finite group with a composition factor isomorphic to a non-abelian simple group $T$, then $$\log k(G) \geq \log k^{*}(T) > c_{1}
{(\log a / \log \log a)}^{1/2}$$ where $a = |\mathrm{Aut}(T)|$.
\end{lem}

From this we may derive the following inequality.

\begin{lem}
\label{l2}
There exists a universal constant $c_{2} > 0$ so that whenever $T$ is a non-abelian finite simple group then
$\log |\mathrm{Aut}(T)| < c_{2} {(\log k^{*}(T))}^{2} \log \log k^{*}(T)$.
\end{lem}

\begin{proof}
From Lemma \ref{l1} we have $\log |\mathrm{Aut}(T)| < (1/{c_{1}}^{2}) {(\log k^{*}(T))}^{2} \log \log |\mathrm{Aut}(T)|$.
From Lemma \ref{l1} we also have that $2 \log \log k^{*}(T) > 2 \log c_{1} + \log \log |\mathrm{Aut}(T)| - \log \log \log |\mathrm{Aut}(T)|$. Notice that this lower bound is non-positive for only at most finitely many $T$'s and it tends to infinity as $|\mathrm{Aut}(T)|$ tends to infinity. Thus $2 \log \log k^{*}(T) > c_{3} \log \log |\mathrm{Aut}(T)|$ for some universal constant $c_{3} > 0$. From these the lemma follows.
\end{proof}
In the next section, we show that $c_2$ can be chosen to be $1.954.$

To slightly simplify notation, for every $i$ with $1 \leq i \leq r$, put $k_{i} = k^{*}(T_{i,j})$ for every $j$ with $1 \leq j \leq n_{i}$.
We may now give an upper bound for $\log |G|$.

\begin{lem}
\label{l3}
Let $c_{2}$ be as above. Then
$\log |G| < n \log n + c_{2} \sum_{i=1}^{r} n_{i} {(\log k_{i})}^{2} (\log \log k_{i})$.
\end{lem}

\begin{proof}
Clearly Lemma \ref{l2} implies $\log |G| < \sum_{i=1}^{r} (n_{i} \log n_{i} + c_{2} n_{i} {(\log k_{i})}^{2} (\log \log k_{i}))$.
\end{proof}

The following lemma will also be useful.

\begin{lem}
\label{l4}
For every $i$ with $1 \leq i \leq r$ the number of conjugacy classes of $G$ lying inside $M_{i}$ is larger than ${(k_{i}/n_{i})}^{n_{i}}$.
\end{lem}

\begin{proof}
Fix an index $i$. Observe that $M_{i}$ has at least $k_{i}^{n_{i}}$ conjugacy classes and that these are non-trivially permuted by a certain
factor group of size at most $n_{i}! < n_{i}^{n_{i}}$.
\end{proof}	

For a permutation group $H$ let $s(H)$ be the number of orbits on the power set of the underlying set. The following is \cite[Theorem 1]{BP}.

\begin{lem}
\label{l5}
Let $H$ be a permutation group of degree $n$. If $H$ has no composition factor isomorphic to ${\Alt}_m$ for $m > t \geq 5$, then $s(H) \geq 2^{c_4 (n/t)}$
for some absolute constant $c_{4} > 0$.
\end{lem}

Let $t \geq 5$ be the largest integer so that $\Alt_t$ is a composition factor of $G/B$. If no such $t$ exists then set $t=4$.
By Lemma \ref{l1} we have $\log k(G) \geq \log k^{*}(\Alt_t)$, provided that $t \geq 5$. If $t \geq 5$ this is at least $c_{5} \sqrt{t}$
by \cite[Lemma 4.3]{Py} for some absolute constant $c_{5} > 0$. Thus in all cases we have $\log k(G) \geq c_{6} \sqrt{t}$ for some other
absolute constant $c_{6} > 0$.

If $t > (\delta^{2}/{c_{6}}^{2}) \cdot {(\log |G|)}^{2/(3+\epsilon)}$ then we are finished. Choose $\delta^{2} < {c_{6}}^{2}$ and assume
that $t< {(\log |G|)}^{2/(3+\epsilon)}$.

By Lemma \ref{l5} we see that $\log k(G) > c_{4} (n/t) > c_{4} (n/{(\log |G|)}^{2/(3+\epsilon)})$. If this is at least
$\delta {(\log |G|)}^{1/(3+\epsilon)}$, then we are finished. So assume that $(c_{4}/\delta) n < {(\log |G|)}^{3/(3+\epsilon)}$.
We may choose $\delta$ smaller than $c_{4}$ so we assume that $n^{1+ (\epsilon/3)} < \log |G|$.

\begin{lem}
\label{l6}
Under our assumptions there exists a constant $c_{7}$ so that $$n^{1+ (\epsilon/3)} < c_{7} \sum_{i=1}^{r} n_{i} {(\log k_{i})}^{2} (\log \log k_{i}).$$
\end{lem}

\begin{proof}
Notice that if $n$ is bounded then we are finished. So assume that $n \rightarrow \infty$.
By our assumption and Lemma \ref{l3} we have $$n^{1+ (\epsilon/3)} < n \log n +  c_{2} \sum_{i=1}^{r} n_{i} {(\log k_{i})}^{2} (\log \log k_{i}).$$ Since $(n \log n) / n^{1 + (\epsilon/3)} \rightarrow 0$ as $n \rightarrow \infty$, there exists a constant $c_{7} > 0$ so that
$$(c_{2}/c_{7}) n^{1+(\epsilon/3)} < n^{1+(\epsilon/3)} - n \log n < c_{2} \sum_{i=1}^{r} n_{i} {(\log k_{i})}^{2} (\log \log k_{i})$$ for large enough $n$. Therefore the proof is complete.
\end{proof}

Set $N(\epsilon)$ to be a large enough integer so that ${(N(\epsilon)/c_{7})}^{1/3} > 2 \log N(\epsilon) \geq 1$ and $m^{\epsilon/18} > 2 \log m$ for all $m$ with $m \geq N(\epsilon)$.

Let $J$ be the set of those $i$'s with $1 \leq i \leq r$ so that $N(\epsilon) \cdot n^{\epsilon/6} < c_{7} {(\log k_{i})}^{2} (\log \log k_{i})$. We may assume that $J$ is non-empty. Otherwise $n$ is bounded by Lemma \ref{l6} and so all the $k_{i}$'s are bounded. This means that $|G|$ is bounded and thus $k(G)$ is bounded. We may set $\delta$ small enough so that the theorem holds for these finitely many groups $G$.

\begin{lem}
\label{l77}
We may assume that there exists a constant $c_{8}$ so that $$\log |G| < c_{8} \sum_{i \in J} n_{i} {(\log k_{i})}^{2} (\log \log k_{i}).$$
\end{lem}

\begin{proof}
By our discussion about $J$ above, our assumption, and Lemma \ref{l3}, we get
$$n^{1+(\epsilon/3)} < \log |G| < n \log n + (c_{2} N(\epsilon)/c_{7}) n^{1+(\epsilon/6)} + c_{2} \sum_{j \in J} n_{i} {(\log k_{i})}^{2} (\log \log k_{i}).$$ Let $K(\epsilon)$ be an integer so that whenever $n \geq K(\epsilon)$ then $$\log |G| - n \log n - (c_{2} N(\epsilon)/c_{7}) n^{1+(\epsilon/6)} > 0.$$ Then there exists a constant $c_{8} > 0$ so that $$(c_{2}/c_{8}) \log |G| < \log |G| - n \log n - (c_{2} N(\epsilon)/c_{7}) n^{1+(\epsilon/6)}$$ whenever $n \geq K(\epsilon)$. Thus we may assume that $n < K(\epsilon)$. Then there exists a positive constant $M(\epsilon)$ so that $$\log |G| < M(\epsilon) + c_{2} \sum_{j \in J} n_{i} {(\log k_{i})}^{2} (\log \log k_{i}).$$ If the second summand on the right-hand side of the previous inequality is larger than $M(\epsilon)$ then the claim follows. Otherwise $n$ and all the $k_{i}$'s are bounded by a constant depending only on $\epsilon$. This means that $|G|$ is bounded. But since $J \not= \emptyset$ we can certainly choose (in this case) a suitable $c_{8}$ to satisfy the statement of the lemma.
\end{proof}

\begin{lem}
\label{l7}
We can assume that for all $i \in J$ we have $\log k_{i} - \log n_{i} > (\log k_{i})/2$.
\end{lem}

\begin{proof}
Since $i \in J$, we have $N(\epsilon) \cdot n^{\epsilon/6} < c_{7} {(\log k_{i})}^{3}$. From this it follows that $${(N(\epsilon)/c_{7})}^{1/3} n^{\epsilon/18} < \log k_{i}.$$ Finally, ${(N(\epsilon)/c_{7})}^{1/3}  n^{\epsilon/18} > 2\log n \geq 2\log n_{i}$ by our choice of $N(\epsilon)$.
\end{proof}

Finally, by Lemmas \ref{l77}, \ref{l7} and \ref{l4}, we have $$\delta^{3} \log |G| < \delta^{3} c_{8} \cdot \sum_{i \in J} n_{i} {(\log k_{i})}^{2} (\log \log k_{i}) < {\Big( (1/2) \sum_{i \in J} n_{i} \log k_{i}\Big)}^{3} <$$ $$<{\Big( \sum_{i \in J} n_{i} (\log k_{i} - \log n_{i}) \Big)}^{3} < {(\log k(G))}^{3}$$ whenever $\delta$ satisfies $\delta^{3} c_{8}  < 1/8$. This proves Theorem~\ref{th:asymptotics}.

\bigskip

\begin{proof}[\textbf{Proof of Theorem \ref{thm1}}]
The proof of Theorem \ref{thm1} depends on Theorem \ref{th:asymptotics}. Indeed, in the proof of \cite[Corollary 3.3]{Keller}, which is an improved version of the argument on \cite[page 248]{Py}, we can replace $7$ by $3+\epsilon$.
Notice that the $\delta$'s in the statements of Theorems \ref{th:asymptotics} and \ref{thm1} are different.
\end{proof}

\section{Computing $c_2$}\label{sec:c2}

Now we turn our attention to Bertram's question aiming to give a specific logarithmic lower bound for $k(G)$ in terms of $|G|$ where $G$ is an arbitrary finite group. In order to prove Theorem \ref{th:base 3}, we need to compute specific values of $c_2$ in Lemma \ref{l2}.

We first fix some notation. Let $T$ be a non-abelian simple group, let $A:=\Aut(T)$ and $k:=k^*(T).$ We have
\begin{equation}
\label{eq1}
k\ge k(T)/|\Out(T)|.
\end{equation} Denote by $\Gamma=\{x_i\}_{i=1}^m$ the representative set for all conjugacy classes of $A,$ i.e., $A=\cup_{i=1}^m x_i^A.$ By definition, we see that \begin{equation}\label{eq2} k= |\{i\in \Gamma:x_i^A\cap T\neq\emptyset\}|.\end{equation}
Notice that $k=k(T)$ when $\Out(T)=1.$
It follows from Lemma \ref{l2} that \begin{equation}\label{eq3} \gamma:= \gamma(T):=\frac{\log |A|}{(\log k)^2 \log\log k}<c_2.\end{equation}

The following lemma is used frequently, whose proof is straightforward and is omitted.

\begin{lem}\label{lem:inequalities} Let $q=p^f\ge 2$ be a power of a prime $p,$ where $f\ge 1$ is an integer and let $2\le a\leq b$ be integers. Then
\begin{enumerate}[$(1)$]
\item $(q^a-1)(q^b-1)\le q^{a+b};$
\item $(q^a-1)(q^b+1)\le q^{a+b};$
\item $q\ge 2f$ and if  $q\ge 16,$ then $q\ge 3f;$
\item If $f\not= 3$, then $2\log f\le f.$
\end{enumerate}
\end{lem}

\begin{thm}\label{th:c2} Let $T$ be a non-abelian simple group. Then $\gamma(T)<1.613$ unless $T\cong \Alt_5$ or $\PSL_3(4).$ For the exceptions, we have $\gamma(\Alt_5)\le 1.727$ and $\gamma(\PSL_3(4))\le 1.954.$
Therefore, we can choose $c_2=1.954$ in all cases. Furthermore, $k\ge 5$ unless $T\cong \Alt_5.$
\end{thm}

For brevity, let $c:=1.613.$ Using \cite[Page xvi]{Atlas}, we can easily obtain Table \ref{tab:Aut}, where $q=p^f$ and $p$ is the defining characteristic of $T$. For `small' simple groups $T$, $k=k^*(T)$ can be computed using \cite{GAP} via the `fusions' of conjugacy classes of $T$ onto that of $\Aut(T)$. Another obvious lower bound for $k^*(T)$ is the number of distinct element orders of $T,$ i.e., \[k=k^*(T)\ge e(T):= |\{|x|\::\:x\in T\}|,\] where $|x|$ denotes the order of the element $x\in T.$

\begin{table}\caption{The finite simple groups of Lie type}\label{tab:Aut}
\begin{tabular}{lllll} \hline
$T$&$d$& $|\Out(T)|$&$|\Aut(T)|\le$\\\hline
$\PSL_n(q)$&$\gcd(n,q-1)$&$2df,n\ge 3$&$2fq^{n^2-1}$\\
&&$df,n=2$&$fq^{3}$\\

$\PSU_n(q)$&$\gcd(n,q+1)$&$2df,n\ge 3$&$2fq^{n^2-1}$\\

$\PSp_{2n}(q)$&$\gcd(2,q-1)$&$df,n\ge 3$&$fq^{2n^2+n}$\\
&&$2f,n=2$&$2fq^{10}$\\

$\Omega_{2n+1}(q), q$ odd &$2$&$2f$&$fq^{2n^2+n}$\\

$\mathrm{P}\Omega_{8}^+(q)$&$\gcd(4,q^4-1)$&$6df$&$2fq^{28}$\\


$\mathrm{P}\Omega_{2n}^+(q)$ &$\gcd(4,q^n-1)$&$2df,n\neq 4$&$2fq^{2n^2-n}$\\

$\mathrm{P}\Omega_{2n}^-(q), n\ge 4$&$\gcd(4,q^n+1)$&$2df$&$2fq^{2n^2-n}$\\

${}^2\mathrm{B}_2(q^2), q^2=2^{2m+1}$&$1$&$2m+1$&$(2m+1)2^{5(2m+1)}$\\

${}^2\mathrm{G}_2(q^2), q^2=3^{2m+1}$&$1$&$2m+1$&$(2m+1)3^{7(2m+1)}$\\

${}^2\mathrm{F}_2(q^2),q^2=2^{2m+1}$&$1$&$2m+1$&$(2m+1)2^{26(2m+1)}$\\

${}^3\mathrm{D}_4(q)$&$1$&$3f$&$6fq^{28}$\\

${}^2\mathrm{E}_6(q)$&$\gcd(3,q+1)$&$2df$&$2fq^{78}$\\

$\mathrm{G}_2(q),q\ge3$&$1$&$f,$ if $p\neq 3$&$fq^{14}$\\
&&$2f,$ if $p= 3$&$2fq^{14}$\\

$\mathrm{F}_4(q)$&$1$&$\gcd(2,p)f$&$\gcd(2,p)fq^{52}$\\

$\mathrm{E}_6(q)$&$\gcd(3,q-1)$&$2df$&$2fq^{78}$\\

$\mathrm{E}_7(q)$&$\gcd(2,q-1)$&$df$&$fq^{133}$\\

$\mathrm{E}_8(q)$&$1$&$f$&$fq^{248}$\\\hline

\end{tabular}

\end{table}

For sporadic and alternating simple groups of small degree, $\gamma(T)$ and $k^*(T)$ are given in Table \ref{tab:alt-spor}. This is done using \cite{GAP}.

\begin{lem}\label{lem:spor-alt}
If $T$ is a sporadic simple group, the Tits group or the alternating group of degree $n\ge 6,$ then $\gamma(T)<c$ while $c<\gamma(\Alt_5)\le 1.727.$ Moreover, $k\ge 5$ unless $T=\Alt_5.$
\end{lem}

\begin{proof}

(i) Assume first that $T$ is a sporadic simple group or the Tits group.
From Table \ref{tab:alt-spor}, we see that $10\le k^*(T)\le k^*(\textrm{M})=194$ and $\gamma(T)\le \gamma(\textrm{M})<1.06<c.$ So, the lemma holds in this case.

(ii) Assume that $T=\Alt_n$ with $5\le n\le 21.$  From Table \ref{tab:alt-spor},  if $6\le n\le 21,$ then $\gamma(T)<1<c$ and $k\ge 5$ while $c<\gamma(\Alt_5)<1.727$ and $k^*(\Alt_5)=4.$

(iii) Assume that $T=\Alt_n$ with $n\ge 22.$ Since $|\SSS_n:\Alt_n|=2,$ Clifford's theorem gives that $k(\SSS_n)\le 2k(\Alt_n)$ and thus by (\ref{eq1}) we have $k\ge {k(\Alt_n)}/{2}\ge {k(\SSS_n)}/{4}={p(n)}/{4},$ where $p(n)$ is the number of partitions of $n$.
By \cite[Corollary 3.1]{M}, we have ${p(n)}/{4}\ge {e^{2\sqrt{n}}}/{56}$ and so, as $n\ge 22,$ we obtain that $k\ge 250$ and  $\log k\ge 2\sqrt{n}\log e-\log 56\ge \sqrt{n}.$ Now we can easily check that
\[\gamma\le \frac{\log n!}{( 2\sqrt{n}\log e-\log 56)^2\log n^{1/2} }< \frac{2n}{(2\sqrt{n}\log e-\log 56)^2}<c. \] This completes the proof.
\end{proof}

\begin{table}\caption{Some alternating and sporadic simple groups}
\begin{tabular}{lll|lll} \hline
$T$&$k $&$\gamma<$ &$T$&$k$&$\gamma<$\\\hline
$\textrm{M}_{11}$&$10$&$0.678$&$\textrm{M}_{12}$&$12$&$0.741$\\

$\textrm{M}_{22}$&$11$&$0.923$&$\textrm{M}_{23}$&$17$&$0.687$\\

$\textrm{M}_{24}$&$26$&$0.565$&$\textrm{J}_{1}$&$15$&$0.581$\\

$\textrm{J}_{2}$&$16$&$0.632$&$\textrm{J}_{3}$&$17$&$0.784$\\

$\textrm{HS}$&$21$&$0.642$&$\textrm{Suz}$&$37$&$0.615$\\

$\textrm{McL}$&$19$&$0.817$&$\textrm{Ru}$&$36$&$0.586$\\

$\textrm{He}$&$26$&$0.668$&$\textrm{Ly}$&$53$&$0.673$\\

$\textrm{O'N}$&$25$&$0.833$&$\textrm{Co}_1$&$101$&$0.511$\\

$\textrm{Co}_2$&$60$&$0.507$&$\textrm{Co}_3$&$42$&$0.550$\\

$\textrm{Fi}_{22}$&$59$&$0.530$&$\textrm{Fi}_{23}$&$98$&$0.519$\\

$\textrm{Fi}_{24}'$&$97$&$0.684$&$\textrm{HN}$&$44$&$0.671$\\

$\textrm{Th}$&$48$&$0.728$&$\textrm{B}$&$184$&$0.678$\\

$\textrm{M}$&$194$&$1.06$&${}^2\textrm{F}_4(2)'$&$17$&$0.740$\\

$\Alt_5$&$4$&$1.727$&$\Alt_6$&$5$&$1.602$\\

$\Alt_7$&$8$&$0.863$&$\Alt_8$&$12$&$0.647$\\

$\Alt_9$&$16$&$0.578$&$\Alt_{10}$&$22$&$0.509$\\

$\Alt_{11}$&$29$&$0.470$&$\Alt_{12}$&$40$&$0.423$\\

$\Alt_{13}$&$52$&$0.399$&$\Alt_{14}$&$69$&$0.374$\\

$\Alt_{15}$&$90$&$0.355$&$\Alt_{16}$&$118$&$0.336$\\

$\Alt_{17}$&$151$&$0.324$&$\Alt_{18}$&$195$&$0.310$\\

$\Alt_{19}$&$248$&$0.300$&$\Alt_{20}$&$\ge 162^*$&$0.395$\\

$\Alt_{21}$&$\ge 204^*$&$0.379$&$\Alt_{22}$&$\ge 256^*$&$0.365$\\\hline
\multicolumn{6}{l}{\mbox{\tiny ${}^*$ We use the bound $k\ge k(\Alt_n)/2.$}} \\
\end{tabular}

\label{tab:alt-spor}
\end{table}

Let $\mathbf{G}$ be a simply connected simple algebraic group of rank $r>0$ and let $F$ be a Steinberg endomorphism of
$\mathbf{G}$ associated to a prime power $q.$ Then $L=\mathbf{G}^F$ is a quasi-simple group and $L/\Center(L)\cong T$
is a finite simple group of Lie type with $d=|\Center(L)|.$ From \cite[Theorem 3.1]{FG} and \cite[Lemma 2.1]{FG}, we
have that $k(L)\ge q^r$ and $k(L)\le k(\Center(L))k(L/\Center(L))$ and thus $k(T)\ge k(L)/k(\Center(L))\ge q^r/d$
hence by (\ref{eq1}), we have

\begin{equation}\label{lower bounds}
k=k^*(T)\ge \textrm{max}\{e(T),\frac{q^r}{d|\Out(T)|}\}.
\end{equation}
Denote by $\Irr(H)$ the set of complex irreducible characters of a finite group $H.$ Then it is well-known that
$k(H)=|\Irr(H)|$ and by Brauer's permutation lemma, the numbers of $\Aut(H)$-orbits on irreducible characters and on conjugacy classes of $H$ are the same.
Therefore, if we write $\cd(H)$ for the set of character degrees of $H,$ then $k^*(H)\ge |\cd(H)|.$ It follows that
\begin{equation}\label{degrees}k^*(T)\ge |\cd(T)|.\end{equation}

\begin{lem}\label{lem:Lie type} Theorem \ref{th:c2} holds for finite simple groups of Lie type.
\end{lem}

\begin{proof}
For the proof of this lemma, we only give a detailed proof for $T=\PSL_n(q)$ with $n\ge 2$ and $q=p^f$ for some prime $p$ and integer $f\ge 1,$
which is the most difficult case. Other families can be dealt with a similar argument.

(i) Assume $T=\PSL_2(q)$ with $q=2^f$. By Lemma \ref{lem:spor-alt}, we can assume that $T$ is not an alternating group. So, $f\ge 3.$ In this case,
we have that $|A|=q(q^2-1)f\le f\cdot 2^{3f}.$ Now, if $3\le f\le 6,$ then $k$ is given in Table \ref{tab2}. For these cases, it is easy to check
that $k\ge 5$ and $$\gamma=\frac{\log |A|}{(\log k)^2\log\log k}\le \frac{3f+\log f}{(\log k)^2}<c.$$ Notice that $1.612006<\gamma(\PSL_2(8))\le 1.613=c.$
We now assume that $f\ge 7.$ We use the lower bound given in \eqref{eq1} where $|\Out(T)|=f$ and $k(T)=q+1$ (see \cite[Theorem 38.2]{Dornhoff}).
So $$k\ge k(T)/|\Out(T)|=(q+1)/f>2^f/f>18.$$ Thus $\gamma\le {(3f+\log f)}/{(f-\log f)^2}.$ Direct computation using the previous inequality shows
that $\gamma<c$ when $f\le 16.$ So, we assume that $f\ge 17.$ Then $f\ge f/2\ge \log f$ and thus $\gamma\le {4f}/{(f-f/2)^2}={16}/{f}<1.$

(ii) $T=\PSL_2(q)$ with $q=7$ or $q=p^f\ge 11$ odd. From \cite[Theorem 38.1]{Dornhoff} we derive that $k(T)=(q+5)/2.$ Moreover, we have
$|A|=q(q^2-1)f$ and $|\Out(T)|=2f.$

(ii)(a) Assume first that $p=3.$ Then $f\ge 3.$ If $f=3,4$ or $5,$ then $k=7,15$ or $27.$ Direct calculation shows that $\gamma<c.$ Assume next that
$f\ge 6.$ We have $k\ge (q+5)/4f\ge 12$ and $\log |A|<\log (fq^3)=3f\log 3+\log f\le 6f$ so $\log k\ge \log(q/4f)=f\log 3-\log (4f)\ge f-2.$
If $f\ge 10,$ then  $\gamma<{6f}/{(f-2)^2}<c.$ So, assume that $6\le f\le 9.$ Then direct calculation using the bound $k\ge (3^f+5)/4f$ confirms
that $\gamma<c.$

(ii)(b)   Assume $p\ge 5$ and $f=1.$ Since $\PSL_2(5)\cong\Alt_5,$ we assume that $p\ge 7.$ Then $\gamma\le {3\log p}/{(\log (p+5)-2)^2}<{3\log p}/{(\log p-2)^2}.$ Clearly, $\gamma<c$ whenever $\log p\ge 6.$ So, assume that $\log p<6$ or equivalently $p<2^6=64$ and hence $p\le 61.$ Now we can check that $\gamma<c$ by using Table \ref{tab2}. If $7\le p\le 71,$ then $k\ge 5$ by Table \ref{tab2}. So, assume $p\ge 71.$ Then $k\ge (p+5)/4\ge 19>5.$

(ii)(c)   Assume $p\ge 5$ and $f=2.$ If $p\le 13,$ then the result follows by using Table \ref{tab2}. So, we assume $p\ge 17.$ Then $k\ge (p^2+5)/8\ge 614$ and $\gamma<(6\log p+1)/(2\log p-3)^2<c$ since $\log p\ge 4.$

(ii)(d) Assume $p\ge 5$ and $3\le f\le 4.$ Then $k=(q+5)/4f>10$ and we can use the same argument as in the previous case to show that $\gamma<c.$

(ii)(e) Assume $p\ge 5$ and $f\ge 5.$ We have $k\ge (q+5)/4f>232$ and  $t=f\log p\ge 11.$ So $\log f\le f\log p/4=t/4$ and \[\gamma<\frac{3f\log p+\log f}{(f\log p-2-\log f)^2}\le \frac{3t+t/4}{(3t/4-2)^2}=\frac{52t}{(3t-8)^2}.\] Since $t\ge 11,$ we see that ${52t}/{(3t-8)^2}<c$ and thus $\gamma<c$ as wanted.

(iii) $T=\PSL_3(q)$ with $q=p^f\ge 3.$ Let $d=\gcd(3,q-1).$ Then $|A|<2fq^8\le q^9,|\Out(T)|=2df$ and $k(T)\ge (q^2+q)/d$ (see \cite{Lubeck}) so $k\ge (q^2+q)/2d^2f.$

(iii)(a) Assume first that $d=\gcd(3,q-1)=3.$ We have $q\ge 2f$ so $k\ge q/9$ and thus $\gamma<{9\log q}/{(\log q-\log 9)^2}\le {9\log q}/{(\log q-3)^2}.$ If $\log q\ge 12,$ then $9\log q/(\log q-3)^2<c$ and $k\ge 819.$ So, assume $\log q<12$ or $q<2^{12}.$

Now if $q=4,$ then $\gamma<1.954;$  if $q=7,$ then $\gamma<c$ by direct calculation using Table \ref{tab2}. If $q=16,$ then $k\ge e(T)=12$ and we get that $\gamma<c.$
Assume that $q\not\in\{4,7,16\}.$ Then $\gamma<c$ by direct calculation using the definition of $\gamma$ with $k\ge (q^2+q)/18f$ and $|A|\le 2fq^8.$ By Table \ref{tab2}, we see that $k\ge 5$ if $q\le 9.$ Assume $q\ge 11.$ If $q/9>4$ or $q>36$ then $k\ge 5.$ So, we may assume $11\le q\le 35.$ Except for $q=16,$ we see that $k\ge (q^2+q)/18f\ge 5.$ For $q=16,$ we can see by \cite{GAP} that $k\ge e(\PSL_3(16))=12.$

(iii)(b) Assume $d=1.$ Here, the argument is similar with $k\ge (q^2+q)/2f\ge q+1>q$ and so $\gamma<(8\log q+\log (2f))/(\log q)^2\le {9}/{\log q}.$
Clearly if $q\ge 53,$ then $9/\log q<c$ and thus $\gamma<c.$ For the remaining values of $q>2,$ direct calculation confirms that $\gamma<c.$ Now, if $q\ge 4,$ then $k\ge q+1\ge 5.$ For the remaining values of $q,$ we see that $k\ge 5.$

(iv) Assume $n\ge 3$ and $q=2.$  Then we may assume that $n\ge 5$ as $\PSL_4(2)\cong\Alt_8$ and $\PSL_3(2) \cong \PSL_2(7)$.
If $n=5,$ then $k=20$ and $\gamma<c.$ So, assume $n\ge 6.$
We have that $d=(n,q-1)=1$ and $f=1$ so $|\Out(T)|=2.$ Hence $k\ge 2^{n-2}\ge 16$ and thus $\gamma<{n^2}/{((n-2)^2\log (n-2))}.$ Since $\log (n-2)\ge \log 4=2,$ we see that \[\frac{n^2}{(n-2)^2\log (n-2)}<\frac{1}{2}(1+\frac{4}{n-2}+\frac{4}{(n-2)^2})\le  \frac{1}{2}(1+\frac{4}{4}+\frac{4}{16})=\frac{9}{8}<c.\]

So, we can assume from now on that $n\ge 4$ and $q\ge 3.$ Then, we have $k(T)\ge q^{n-1}/d$ (see \cite[Corollary 3.7]{FG}) and thus $k\ge {q^{n-1}}/{(2d^2f)}\ge {q^{n-2}}/{d^2}\ge {q^{n-3}}/{d}\ge q^{n-4}.$  Therefore
\begin{equation}\label{eqn1}\gamma<\frac{(n^2-1)\log q+\log (2f)}{((n-1)\log q-\log (2d^2f))^2\log ((n-1)\log q-\log (2d^2f))}\end{equation}

or \begin{equation}\label{eqn2}\gamma<\frac{(n^2-1)\log q+\log (2f)}{((n-2)\log q-2\log d)^2\log ((n-2)\log q-2\log d)}.\end{equation}

(v) Assume $4\le n\le 7$ and $q\ge 3.$ We can use the same argument as in Case (iii) above to obtain the result.
As an example, assume that $n=4.$ We deduce from Inequality \eqref{eqn2} that
$$\gamma <\frac{15\log q+\log (2f)}{(2\log(q)-2\log d)^2}\le \frac{4\log q}{(\log q-\log d)^2}\le \frac{4\log q}{(\log q-2)^2}.$$ We see that $4\log q/(\log q-2)^2<c$ whenever $\log q\ge 6$ and thus $\gamma<c.$ For all $q\ge 3$ with $\log q<6$ or equivalently $q<2^8=256,$ direct calculation using Equation \eqref{eqn1} shows that $\gamma<c.$ Since $k\ge q^2/d^2\ge q^2/16,$ we see that $k\ge 5$ if $q>8.$ For $3\le q\le 8,$ we can check directly that $k\ge 5.$

(vi) Assume $n\ge 8$ and $q\ge 3.$ Then $k\ge q^{n-4}\ge 81,$ \[\frac{n^2}{(n-4)^2}=(1+\frac{4}{n-4})^2=1+\frac{8}{n-4}+\frac{16}{(n-4)^2}\le 4\] and $\log((n-4)\log q)\ge \log 4=2.$ From Inequality \eqref{eqn1}, we have that  \[\gamma< \frac{n^2}{(n-4)^2\log q\log((n-4)\log q)}\le \frac{4}{2\log q}<c.\]
This completes the proof.
\end{proof}

The proof of Theorem \ref{th:c2} now follows by combining Lemmas \ref{lem:spor-alt} and \ref{lem:Lie type}.

\begin{table}\caption{$\PSL_2(q)$ and $\PSL_3(q)$ with small $q$}
\begin{tabular}{lll|lll} \hline
$T$&$k $&$\gamma<$& $T$&$k$&$\gamma<$\\\hline

$\PSL_2(8)$&$5$&$1.613$& $\PSL_2(16)$&$7$&$1.193$\\
$\PSL_2(32)$&$9$&$1.036$& $\PSL_2(64)$&$15$&$0.686$\\
$\PSL_2(7)$&$5$&$1.281$& $\PSL_2(11)$&$7$&$0.884$\\
$\PSL_2(13)$&$8$&$0.778$& $\PSL_2(17)$&$10$&$0.642$\\
$\PSL_2(19)$&$11$&$0.595$& $\PSL_2(23)$&$13$&$0.525$\\
$\PSL_2(25)$&$10$&$0.782$& $\PSL_2(27)$&$7$&$1.351$\\
$\PSL_2(29)$&$16$&$0.456$& $\PSL_2(31)$&$17$&$0.438$\\
$\PSL_2(37)$&$20$&$0.397$& $\PSL_2(41)$&$22$&$0.375$\\
$\PSL_2(43)$&$23$&$0.366$& $\PSL_2(47)$&$25$&$0.349$\\
$\PSL_2(49)$&$17$&$0.526$& $\PSL_2(53)$&$28$&$0.329$\\
$\PSL_2(59)$&$31$&$0.312$& $\PSL_2(61)$&$32$&$0.307$\\
$\PSL_2(67)$&$35$&$0.294$& $\PSL_2(71)$&$37$&$0.286$\\
$\PSL_2(121)$&$37$&$0.337$& $\PSL_2(169)$&$50$&$0.292$\\

$\PSL_3(4)$&$6$&$1.954$& $\PSL_3(7)$&$15$&$0.781$\\

$\PSL_3(3)$&$9$&$0.805$& $\PSL_3(5)$&$19$&$0.518$\\

$\PSL_3(8)$&$17$&$0.783$& $\PSL_3(9)$&$32$&$0.471$\\
\hline
\end{tabular}

\label{tab2}
\end{table}

\section{Almost simple groups}\label{sec:almost simple}

In this section, we prove the following.

\begin{thm}\label{th:almost simple}
Let $G$ be an almost simple group with non-abelian simple socle $T.$ Suppose that $k=k^*(T)\le 153.$ Then \begin{equation}\label{base3}\log|G|\le (\log 3)k(G).\end{equation}
\end{thm}

\begin{proof}
We now describe our strategy for the proof of this theorem.
We consider the following setup. Let $T$ be a non-abelian simple group and, $A:=\Aut(T)$ and $k=k^*(T).$ Let $G$ be an almost simple group with socle $T,$ i.e., $T\unlhd G\le A.$

Firstly, if $T$ is a sporadic simple group, the Tits group or an alternating group of degree at most $22,$ then the result follows by direct computation with \cite{GAP} or \cite{Magma}. For $T=\Alt_n$ with $n\ge 23,$ it follows from the proof of Lemma \ref{lem:spor-alt} that $k^*(T)\ge 250>153.$ So, we may assume that $T$ is a finite simple group of Lie type.

Now suppose that $T$ is of Lie rank $r$ and defined over a field of size $q.$ Let $d$ be defined as in Section \ref{sec:c2}. Then we know that $k(T)\ge q^r/d$ and thus $k=k^*(T)\ge q^r/(d|\Out(T)|).$ We now use the restriction $k\le 153$ to obtain a finite list $\mathcal{L}$ of all simple groups $T$ with $k\le 153.$ Since $k(G)\ge k^*(T)$ by \cite[Lemma~2.5]{Py} and $\log |G|\le \log |A|,$ if we can show that
\begin{equation}\label{auto} \log |A|\le (\log 3) k
\end{equation}
then obviously Inequality \eqref{base3} holds.
For the remaining groups, we can check Inequality \eqref{base3} directly using the known bound for $k(T)$ or using \cite{Magma,GAP,Lubeck}.

For the purpose of computation, the following observation will be useful.
Suppose that $A:=\Aut(T)=\Gamma \la \tau\ra,$ where $T\unlhd \Gamma\le A$ with $|A:\Gamma|=s$ for some integer $s\ge 1.$
Now, if we can prove that for every almost simple group $G$ with $T\unlhd G\leq \Gamma,$ we have $s\cdot |G|\le 3^{k(G)/s},$
then $|H|\le 3^{k(H)}$ for all almost simple groups with socle $T.$ This follows from the fact that if $T\unlhd H\le A,$
then $G:=T\cap \Gamma$ has index at most $s$ in $H,$ so $|H|\le s|G|$ and $k(H)\ge k(G)/s.$ Therefore, if $s\cdot |G|\le 3^{k(G)/s},$
then obviously $3^{k(H)}\ge 3^{k(G)/s}\ge s|G|\ge |H|$ as wanted. This will be useful when we can compute $k(G)$ for all $T\unlhd G\le \Gamma.$
This observation applies when, for example, $T=\PSL_n(q),(n\ge 3),$ $\Gamma=\textrm{P}\Gamma \textrm{L}_n(q)$ and $A=\Gamma\la\tau\ra,$
where $\tau$ is a graph automorphism of $T$ of order $2.$

To demonstrate our strategy, we give a detailed proof for $T=\PSL_n(q)$ with $n\ge 2$ and $q=p^f.$

(i) Assume that $n=2.$ Suppose first that $q=2^f.$ From \cite[Theorem 38.2]{Dornhoff}, we have $k(T)=2^f+1.$  Using \cite{GAP},
we can check that the result holds for $2\le f\le 7.$ Assume $f > 7.$ Since $153\ge k\ge (2^f+1)/f,$ we deduce that $7< f\le 11.$ We have that
$\log |G|\le \log |\Aut(T)|\le 3f+\log f.$ If $G=T,$ then the result is obvious, so we may assume $G\neq T.$ We now can use \cite{Magma}
to show that Inequality \eqref{base3} holds for all almost simple groups $G$ with socle $T=\PSL_2(2^f)$, with $7< f\le 11.$

Assume next that $q=p^f\ge 7$ is odd. Then $k(T)=(q+5)/2$ and $k(\PGL_2(q))=q+2$, see \cite[Theorem 38.1]{Dornhoff}.  Clearly, we can check that the result holds in these cases. So, we may assume from now on that $G\not\cong \PSL_2(q)$ nor $\PGL_2(q).$ Moreover, if $f\ge 5$ and $p\ge 5,$ then $k\ge (p^{2f}+5)/(4f)\ge (5^{2f}+5)/(4f)\ge  154.$ So, we only need to consider the following cases.

If $f=1,$ then $q=p^f=p\ge 5.$ Since $153\ge k\ge (p+5)/4,$ we have $p\le 607.$ So, $G=\PSL_2(p)$ or $\PGL_2(p)$ with $p\le 607$ and the result follows
using \cite{Magma}.

If $f=2,$ then, arguing as above, we obtain that $p\le 31.$ Similarly, if $f=3,$ then $p\le 11$ and finally, if $f=4,$ then $p\le 7.$ Now, we can use \cite{Magma} to verify that Inequality \eqref{base3} holds in these cases.

(ii) Assume that $q=2$ and $n\ge 3.$ Then $k\ge 2^{n-1}/(2d^2f)=2^{n-2}$ as $f=d=1.$ Since  $k\le 153,$ we have $n\le 9.$ Now, if $n\ge 7,$ then $(\log 3)k\ge (\log 3)2^{n-2}\ge n^2\ge \log |A|,$ hence Inequality \eqref{auto} holds and so Inequality \eqref{base3} holds in this case.  For $3\le n\le 6,$ we can check directly that Inequality \eqref{base3} holds using \cite{GAP}.

(iii) Assume that $q=3$ and $n\ge 3.$ Then $d=\gcd(n,q-1)=\gcd(n,2)\le 2$ and $f=1,$ so $k\ge 3^{n-1}/(2d^2f)=3^{n-1}/8.$ Since $k\le 153,$ we have $n\le 7.$

If $n=7,$ then $d=\gcd(7,2)=1$ so $(\log 3) k\ge (\log 3)3^{n-1}/2>n^2\log 3>\log |A|$ and thus Inequality \eqref{auto} holds.

If $n=6,$ then $d=\gcd(6,2)=2$ and $k(T)=204$ by \cite{Magma}. So $k\ge k(T)/|\Out(T)|\ge 51.$ Now we can check that $\log |A|<36 \log 3<51\log 3<(\log 3)k.$

If $n=5,$ then $d=\gcd(5,2)=1$ and $k(T)=116,$ so $k\ge 116/2=58.$ Hence Inequality \eqref{auto} holds.

Finally, if $n=3,4,$ then the results follow by using \cite{Magma}.

(iv) Assume now that $n=3$ and $q\ge 4.$ We have that $k(T)\ge (q^2+q)/d$ and thus $k\ge (q^2+q)/(2d^2f)\ge (q+1)/d^2\ge (q+1)/9.$ Since $k\le 153,$ we have $q\le 1376.$ For these values of $q,$ we can check directly that $\log|A|\le 9\log q\le (\log 3)(q^2+q)/(2d^2f)\le k(\log 3)$ unless $q\in\{4,7,8,13,16,19,25\}.$

The case when $q=4,7,8$ can be check directly using \cite{GAP}. For $q=16,25,$ we can check that $2|G|\le 3^{k(G)/2}$ for all $T\unlhd G\le \Gamma$ and thus the results follow by the observation above.

(v) Assume that $n=4$ and $q\ge 4.$ Since $d=\gcd(n,q-1)\le n=4$ and $153\ge k\ge q^3/(2d^2f)\ge q^2/16,$ we deduce that $4\le q\le 49.$ However, we can check that $\log |A|<36\log q<(\log 3)q^3/(2d^2f)\le (\log 3)k\le (\log 3) k$ unless $q=4,5,9.$ Now we use the observation and \cite{Magma} to show that Inequality \eqref{base3} holds for the remaining cases.

(vi) Assume that $n=5$ and $q\ge 4.$ Since $d=\gcd(n,q-1)\le 5$ and $153\ge k\ge q^4/(2d^2f)\ge q^3/25,$ we deduce that $q\le 15,$ so $q=4,5,7,8,9,11,13.$
However, we can check with \cite{Magma} that $\log |A|<25\log q<(\log 3)q^4/(2d^2f)\le (\log 3)k,$ so Inequality \eqref{auto} holds.

(vii) Assume that $n=6$ and $q\ge 4.$ Since $d=\gcd(n,q-1)\le 6=n$ and $153\ge k\ge q^5/(2d^2f)\ge q^4/36,$ we deduce that $q\le 8,$ so $q=4,5,7,8.$ However, we can check that $\log |A|<36\log q<(\log 3)q^5/(2d^2f)\le (\log 3)k$ unless $q=4.$ Now we use the observation together with \cite{Magma} to verify \eqref{base3} for this case.

(viii) Assume that $n=7$ and $q\ge 4.$ Since $d=\gcd(n,q-1)\le n=7$ and $153\ge k\ge q^6/(2d^2f)\ge q^5/49,$ we deduce that $q\le 5,$ so $q=4,5.$ However, we can check that $\log |A|<49\log q<(\log 3)q^6/(2d^2f)\le (\log 3)k.$

(ix) Assume that $n\ge 8$ and $q\ge 4.$ We see that $k\ge q^{n-1}/(2d^2f)\ge q^{n-2}/d^2\ge q^{n-4}\ge 4^4=256>153.$ So this case cannot occur.
 \end{proof}

\section{More on groups with a trivial solvable radical}\label{sec:trivial radical}

This section is devoted to proving Theorem \ref{th:base 3}. We use the notations and assumptions of Section \ref{sec:asymptotics}. We start with a lemma.

\begin{lem}
\label{l8}
With the notation and assumption in Section \ref{sec:asymptotics}, we have $$\prod_{i=1}^{r} \binom{n_{i}+k_{i}-1}{k_{i}-1} \leq k(G).$$
\end{lem}

\begin{proof}
It is sufficient to show that the number of orbits of $G$ on the set of conjugacy classes of $N$ is at least $\prod_{i=1}^{r} \binom{n_{i}+k_{i}-1}{k_{i}-1}$. For this it is sufficient to show that the number of orbits of $G$ on the set of conjugacy classes of $M_{i}$ (for any fixed $i$ with $1 \leq i \leq r$) is at least $\binom{n_{i}+k_{i}-1}{k_{i}-1}$. But this follows from \cite[Lemma 2.6]{F} since we may assume that $G$ is as large as possible, that is, it induces an action of $S_{n_i}$ on the factors of $M_{i}$.
\end{proof}

We continue with another lemma.

\begin{lem}\label{lem1}
Let  $4\le k\in\N.$ Then $(\log k)^2\log\log k\le {k^2}/{2}.$
\end{lem}
\begin{proof}
Let $x=\log k\ge 2.$ Then  $\log\log k\le \log k$ and hence, it suffices to prove that $4^{x}\ge 2x^3$ which is always true when $x\ge 5.$ For $2\le x<5$ or $4\le k<32,$ we can check directly that the inequality in the lemma holds true.
\end{proof}

Consider the inequality
\begin{equation}\label{eqn}n_{i} \log n_{i} + c_{2} n_{i} {(\log k_{i})}^{2} (\log \log k_{i}) \leq w_{i} \cdot (\log 3) \binom{n_{i}+k_{i}-1}{k_{i}-1}\end{equation}
for a fixed positive number $w_{i}$.

\begin{lem}\label{lem:n=1} In Inequality \eqref{eqn}, let $n=n_i\ge 1,$  $k=k_i\ge 4$, $c_2=1.954$, and let $w = w_{i}$. Then
\begin{enumerate}
\item[$(i)$] If $n=1$ and $k\ge 222,$ then Inequality \eqref{eqn} holds with $w=1$.
\item[$(ii)$] If $n=2$ and $k\ge 9,$ then Inequality \eqref{eqn} holds with $w=1$.
\item[$(iii)$] Inequality \eqref{eqn} always holds with $w=1$ if $n\ge 3$.

\item[$(iv)$] If $n=2$ and $4 \leq k < 9$, then Inequality \eqref{eqn} holds with $w=1.17$.
\item[$(v)$] If $n=1$ and $k\ge 4,$ then Inequality \eqref{eqn} holds with $w=2.5$.
\end{enumerate}
\end{lem}
\begin{proof}
(i) Assume that $n=1$ and $w=1$. Then Inequality \eqref{eqn} is equivalent to \begin{equation}\label{eq4}c_2(\log k)^2\log\log k\le k\log 3.\end{equation}
Since $k\ge 4,$ we see that $\log k\le k$ and so $\log \log k\le \log k.$ Hence $c_2(\log k)^2\log\log k\le c_2(\log k)^3.$ Thus it suffices to show that $c_2(\log k)^3\le (\log 3)k$ or \[2^x\ge c_2x^3/\log 3\] where $x=\log k.$ Clearly, we can see that this inequality holds when $x\ge 11$ or equivalently $k\ge 2^{11}.$ For $k<2^{11},$ we can check that Inequality \eqref{eq4} holds provided that $k\ge 222.$

(ii) Assume that $n=2$ and $w=1$. Then Inequality \eqref{eqn} is equivalent to \begin{equation}\label{eq5} 2+2c_2(\log k)^2\log\log k\le (\log 3)k(k+1)/2.\end{equation}
Observe that $ 2+2c_2(\log k)^2\log\log k\le 2+2c_2(\log k)^3$ and $(\log 3)k(k+1)/2\ge 3k^2/4.$ So it suffices to show that $3k^2/4\ge 2+2c_2(\log k)^3.$ We can see that this inequality holds true when  $k\ge 32.$ For $4\le k<31,$ we can check that Inequality \eqref{eq5} holds only when $k\ge 9.$

(iii) Assume that $n\ge 3.$
Suppose first that $n=3.$ Arguing as in (ii), we see that Inequality \eqref{eqn}  is equivalent to \begin{equation}\label{eq6} 3\log 3+3c_2(\log k)^2\log\log k\le (\log 3)k(k+1)(k+2)/6.\end{equation}
Observe that $$ 3\log 3+3c_2(\log k)^2\log\log k\le 6+3c_2(\log k)^3$$ and $$(\log 3)k(k+1)(k+2)/6\ge k^3/4.$$ So it suffices to show that $k^3/4\ge 6+3c_2(\log k)^3.$  Clearly, the latter inequality holds true when $k\ge 8.$ For $4\le k<8,$ we can check directly that Inequality \eqref{eq5} holds.
The same argument can be applied for $n=4,5$ to show that Inequality \eqref{eqn} holds.

 So, assume that $n\ge 6.$ Assume next that $k=4.$ Then Inequality \eqref{eqn}  is equivalent to \begin{equation}\label{eq7} n\log n+4c_2n\le (\log 3)(n+1)(n+2)(n+3)/6.\end{equation}
Since $n\log n+4c_2n\le n^2+8n$ and $$\binom{n+3}{3}\log 3\ge (n+3)(n+2)(n+1)/4,$$ to prove Inequality \eqref{eq7}, it suffices to show that $4n(n+8)\le (n+3)(n+2)(n+1)$ which is always true as $n\ge 6.$
Therefore, one can assume that $n\ge 6$ and $k\ge 5.$

Since $k-1\ge 4,$ we deduce that $$\binom{n+k-1}{n}=\binom{n+k-1}{k-1}\ge \binom{n+k-1}{4}.$$ Hence, as $\log 3\ge 3/2,$ we have \[(\log 3)\binom{n+k-1}{k-1}\ge \frac{3}{2}\binom{n+k-1}{4}=\frac{(n+k-1)(n+k-2)(n+k-3)(n+k-4)}{16}.\]

Since $(\log k)^2\log\log k\le k^2/2$ by Lemma \ref{lem1} and $\log n\le n,$ we deduce that \[n \log n + c_{2} n {(\log k)}^{2} (\log \log k)\le n^2+nk^2.\]
Therefore, it suffices to show that \begin{equation}\label{eq8}{(n-1+k)(n-2+k)(n-3+k)(n+k-4)}\ge 16n(n+k^2).\end{equation}
Since $n+k-4\ge n,$ to prove \eqref{eq8}, it suffices to prove that \begin{equation}\label{eq9}{(n-1+k)(n-2+k)(n-3+k)}\ge 16n+16k^2.\end{equation}
We have that \[(n-1+k)(n-2+k)(n-3+k)=(n-1)(n-2)(n+k-3)+(2n-3)k(n+k-3)+k^2(n+k-3).\]
Since $n+k-3\ge n\ge 6,$ we have \begin{equation}\label{eq10}(n-1)(n-2)(n+k-3)\ge 5\cdot 4\cdot n=20n>16n.\end{equation}
Since $n+k-3\ge k\ge 5$ and $n\ge 6,$ we have \begin{equation}\label{eq11}k(2n-3)(n+k-3)\ge 9k^2
\end{equation}
and \begin{equation}\label{eq12}(n+k-3)k^2\ge (6+5-3)k^2=8k^{2}.
\end{equation}
Adding \eqref{eq11} and \eqref{eq12}, we obtain that \begin{equation}\label{eq13}k(2n-3)(n+k-3)+(n+k-3)k^2\ge 17k^2>16k^{2}.
\end{equation}
Now \eqref{eq9} follows by adding \eqref{eq10} and \eqref{eq13}.

Finally, (iv) and (v) can be checked using a computer.
\end{proof}

Using the information from Lemma \ref{lem:n=1} we define numbers $w_{i}$ for each $i$ with $1 \leq i \leq r$. If $n_{i} = 1$ and
$4 \leq k_{i} < 222$, then put $w_{i} =2.5$. If $n_{i}=2$ and $4 \leq k_{i} < 9$, then put $w_{i} = 1.17$. In all other cases put $w_{i} = 1$.
We need another lemma.

\begin{lem}
\label{Lem1}
Let $r$ be a positive integer and let $x_{1}, \ldots, x_{r}$ be integers which are at least $4$. Then the following are true.
\begin{enumerate}
\item[$(i)$] If $r \geq 3$ then $2.5 \cdot \sum_{i=1}^{r} x_{i} \leq \prod_{i=1}^{r} x_{i}$.

\item[$(ii)$] If $r = 2$ then $2.5 x_{1} + 1.17 x_{2} \leq x_{1} x_{2}$.
\item[$(iii)$] If $r=2$ and $x_i\ge 5,$ then $2.5 x_1+2.5x_2\le x_1x_2.$
\end{enumerate}
\end{lem}

\begin{proof}
(i) can be seen by induction on $r$. (ii) and (iii) are  easy computations.
\end{proof}

\begin{proof}[\textbf{Proof of Theorem \ref{th:base 3}}]
By Lemmas \ref{l3} and \ref{lem:n=1}, we have $$\log |G| < \sum_{i=1}^{r} \Big( n_{i} \log n_{i} + c_{2} n_{i} {(\log k_{i})}^{2} (\log \log k_{i}) \Big) \leq (\log 3) \sum_{i=1}^{r} w_{i} \binom{n_{i}+k_{i}-1}{k_{i}-1}.$$ By Lemma \ref{Lem1} and the fact that the binomial coefficients we consider are all at least $4$ (since $k_{i} \geq 4$ and $n_{i} \geq 1$ for every $i$ with $1 \leq i \leq r$), this is at most
$$\leq (\log 3) \prod_{i=1}^{r} \binom{n_{i}+k_{i}-1}{k_{i}-1} \leq (\log 3) k(G)$$ where the last inequality follows from Lemma \ref{l8}, unless possibly if one of the following cases holds.

\begin{enumerate}
\item $r = 1$, $n_{1} = 1$ and $4 \leq k_{1} < 222$;

\item $r = 1$, $n_{1} = 2$ and $4 \leq k_{1} < 9$; or

\item $r = 2$, $n_{1} = n_{2} = 1$ and $k_{1}= k_{2} =4$.
\end{enumerate}

In all cases the group $G$ has a socle which is the product of at most two non-abelian simple groups.

Case $r=1$ and $n_1=2.$ Observe that when $n_1=2,$ then Inequality \eqref{eqn} holds for simple groups $T$ with $\gamma(T)\le 1.613$  and $w_1=1.$ So, $\log |G|<(\log 3) k(G)$ whenever $\textrm{Soc}(G)\cong T^2$ and $T\not\cong \PSL_3(4),\Alt_5.$ For the remaining cases, we see that $$\textrm{Soc}(G)\cong T^2\unlhd G\le\Aut(T^2)\cong\Aut(T)\wr\Sym(2).$$ Now using \cite{Magma}, we can check that $\log |G|\le (\log 3)k(G).$

Case $r=2,$ $n_1=n_2=1$ and $k_1=k_2=4.$ Then $\textrm{Soc}(G)\cong T_1\times T_2$ and $$T_1\times T_2\unlhd G\le \Aut(T_1)\times \Aut(T_2),$$ where $T_i$ is a non-abelian simple group with $k_i=k^*(T_i)=4$ for $i =1$ and $2$. It follows from Theorem \ref{th:c2} that $T_i=\Alt_5$ with $i=1,2.$ Hence $\Alt_5^2\unlhd G\le \SSS_5\times \SSS_5.$ Using \cite{Magma} again, it is routine to check that $\log |G|\le (\log 3)k(G).$

Therefore, we are left with the case $r=1,n_1=1$ and $4\le k_1\le 221.$ So, $G$ is an almost simple group with non-abelian simple socle $T$ and
$4\le k=k^*(T)\le 221.$ Clearly, $\log |G|\le (\log 3)k(G)$ if $T\cong \Alt_5$ or $\PSL_3(4).$ So, we may assume that $T$ is not one of those groups.
Then $\gamma(T)<1.613$ by Theorem \ref{th:c2} ($\gamma(T)$ is defined in Section \ref{sec:c2}). We can now bound $k_1$ by $153$
(see the proof of Lemma \ref{lem:n=1}(i)).
We obtain the inequality $\log_{3} |G| \leq  k(G)$ by applying  Theorem~\ref{th:almost simple}.
As by assumption $G$ is not a $3$-group, the latter is a strict inequality and the result follows.
\end{proof}

\end{document}